\documentclass[a4paper,12pt]{amsart}

\usepackage[latin1]{inputenc}                

\usepackage{amsthm,amsfonts,amsmath,amssymb} 
\usepackage{mathrsfs}                        
\usepackage{dsfont}                          

\usepackage{hyperref}                        

\usepackage[all]{xy}                         

\newcommand{\ud}{\mathrm{d}}

\newcommand{\e}{\mathrm{e}}
\newcommand{\CR}{\mathds{R}}

\newcommand{\dev}[2]{\displaystyle\frac{\ud\emph{$#1$}}{\ud\emph{$#2$}}}

\newcommand{\fraction}{\displaystyle\frac}

\newcommand{\somatorio}[2]{\sum\limits_{\emph{$#1$}}^{\emph{$#2$}}}

\theoremstyle{plain}
\newtheorem{teo}{Theorem}[section]
\newtheorem{defi}{Definition}[section]
\newtheorem{lem}{Lemma}[section]
\newtheorem{prop}{Proposition}[section]
\newtheorem{coro}{Corollary}[section]

\theoremstyle{definition}

\newtheorem{rema}{Remark}[section]

\numberwithin{equation}{section}

{\left\{\begin{array}{ll}} {\end{array}\right.}

\begin{document}

\title{On a Poincar\'{e} lemma for foliations}
\author{Eva Miranda}\address{ Eva Miranda,
Departament de Matem\`{a}tica Aplicada I, Universitat Polit\`{e}cnica de Catalunya, EPSEB, Avinguda del Doctor Mara\~{n}\'{o}n, 44-50, 08028, Barcelona, Spain, \it{e-mail: eva.miranda@upc.edu}}

 \author{Romero Solha}
 \address{ Romero Solha,
Departament de Matem\`{a}tica Aplicada I, Universitat Polit\`{e}cnica de Catalunya, ETSEIB, Avinguda Diagonal 647, 08028, Barcelona, Spain, \it{e-mail: romero.barbieri@upc.edu}}

\thanks{Both authors have been partially supported by the DGICYT/FEDER project MTM2009-07594: Estructuras Geometricas: Deformaciones, Singularidades y Geometria Integral until December 2012 and by the MINECO project GEOMETRIA ALGEBRAICA, SIMPLECTICA, ARITMETICA Y APLICACIONES with reference: MTM2012-38122-C03-01  starting in January 2013. This research has also been partially supported by ESF network CAST, \emph{Contact and Symplectic Topology}. Romero Solha has been partially supported by Start-Up Erasmus Mundus External Cooperation Window
2009-2010 project.}
\date{\today}


\begin{abstract} In this paper we  revisit a Poincar\'{e} lemma for foliated forms, with respect to a regular foliation, and compute the foliated cohomology for local models of integrable systems with singularities of  nondegenerate type. A key point in this computation is the use of some analytical tools for integrable systems with nondegenerate singularities, including a Poincar\'{e} lemma for the deformation complex associated to this singular foliation.
\end{abstract}


\maketitle

\section{Introduction}

In \cite{MiNg} Vu Ngoc and the first author of this paper proved a \emph{singular Poincar\'{e} lemma} for the deformation complex of an integrable system with nondegenerate singularities. This complex is the Chevalley-Eilenberg complex \cite{chevalley-eilenberg} associated to a representation by Hamiltonian vector fields of this integrable system on the set of functions (modulo basic functions). The initial motivation for \cite{MiNg}  was to give a complete proof for a crucial lemma used in proving a deformation result for pairs of local integrable systems with compatible symplectic forms. This deformation proves a Moser path lemma which is a key point in establishing symplectic normal forms \emph{à la Morse-Bott} for integrable systems with nondegenerate singularities \cite{eli1,eli2,Mi}. This normal form proof can be seen as a \emph{``infinitesimal stability theorem implies stability''} result in this context (see \cite{evaintegrablegroup}). So the Poincar\'{e} lemma turns out to be an important ingredient  in the study of the Symplectic Geometry of integrable systems with singularities.

In this paper we use the Poincar\'{e} lemma of the deformation complex to compute some cohomology groups associated to the singular foliation defined by the Hamiltonian vector fields of an integrable system. In particular, we consider the analytic case, in which this computation becomes simpler and can be done in full generality.

A Poincar\'{e} lemma exists when the foliation is regular, and an offspring of this is a Poincar\'{e} lemma in the context of Geometric Quantization, in which the considered complex is a twisted complex from foliated cohomology: the \emph{Kostant complex}. This Poincar\'{e} lemma turns out to be handy because it allows to compute a sheaf cohomology associated to Geometric Quantization. We enclose a sketch of the proof of these two Poincar\'{e} lemmata.
	
If we consider singularities into the picture, the whole scenario changes. As concerns the analytical tools, what makes the difference between the regular and singular case are the solutions of the equation $X(f)=g$ for a given $g$ and a given vector field $X$. When the vector field is regular, we can solve this equation by simple integration no matter which function $g$ is considered. If the vector field is singular, then this is a nontrivial question. Solutions may exist or not depending on some properties of the function $g$ and the singularity of the vector field $X$. For instance, solutions of this equation are studied in \cite{guilleminfuchsian}.

The nonexistence of solutions of equations of type $X(f)=g$ are interpreted in this paper as an obstruction for local solvability of the cohomological equation $\ud_\mathcal{F}\beta=\alpha$, for a given foliated closed $k$-form $\alpha$. Indeed, the fact that the vector fields defining the foliation commute adds an additional ingredient for the simultaneous solution of several equations of this type, which was already exploited in \cite{MiNg} and is further studied in this paper.

{\bf{Organization of this paper:}} In section 2 we describe the geometry of the singular foliations considered in this paper. We recall in section 3 the singular Poincar\'{e} lemma for a deformation complex contained in \cite{MiNg}. We revisit in section 4 the proof of the regular Poincar\'{e} lemma using homotopy operators provided in \cite{geometricasymptotics}, indicate how to apply these techniques to prove a Poincar\'{e} lemma for regular foliations and show an application to Geometric Quantization. In section 5 we consider the case when the foliation given by the integrable system has rank $0$ singularities and compute the foliated cohomology groups.


\section{Singular foliations given by nondegenerate integrable systems}

An integrable system on a symplectic manifold $(M,\omega)$ of dimension $2n$ is a set of $n$ functions $f_1,\dots,f_n\in C^\infty(M)$ satisfying $\ud f_1\wedge\cdots\wedge\ud f_n\neq 0$ over an open dense subset of $M$ and $\{f_i,f_j\}=0$ for all $i,j$. The mapping $F=(f_1,\dots,f_n):M\to\CR^n$ is called a moment map.

The Poisson bracket is defined by $\{f,g\}=X_f(g)$, where $X_f$ is the unique vector field defined by $\imath_{X_f}\omega=-\ud f$: the Hamiltonian vector field of $f$.

The distribution generated by the Hamiltonian vector fields of the moment map, $\langle X_{f_1},\dots,X_{f_n}\rangle$, is involutive because $[X_f, X_g]=X_{\{f,g\}}$. Since $0=\{f_i,f_j\}=\omega(X_{f_i}, X_{f_j})$, the leaves of the associated (possibly singular) foliation are isotropic submanifolds; they are Lagrangian at points where the functions are functionally independent.\par

There is a notion of nondegenerate singular points which was initially introduced by Eliasson \cite{eli1,eli2}. We may consider different ranks for the singularity. To  define the $k$-rank case we reduce to the $0$-rank case considering a Marsden-Weinstein reduction associated to a natural Hamiltonian $\mathds{T}^k$-action \cite{Zu1,mirandazungequiv} given by the joint flow of the moment map $F$.

We denote by $(x_1,y_1,\dots,x_n,y_n)$ a set of coordinates centred at the origin of $\CR^{2n}$ and by $\somatorio{i=1}{n} \ud x_i\wedge \ud y_i$ the Darboux symplectic form.

In the rank zero case, since the functions $f_i$ are in involution with respect to the Poisson bracket, their quadratic parts commute, defining in this way an Abelian subalgebra of $Q(2n,\CR)$ (the set of quadratic forms on $2n$-variables). These singularities are said to be of nondegenerate type if this subalgebra is a Cartan subalgebra.

Cartan subalgebras of $Q(2n,\CR)$ were classified by Williamson in \cite{williamson}.

\begin{teo}[Williamson]{\label{Willi}}
  For any Cartan subalgebra $\mathcal H$ of $Q(2n,\CR)$ there is
  a symplectic system of coordinates $(x_1,y_1,\dots,x_n,y_n)$ in
  $\CR^{2n}$ and a basis $h_1,\dots,h_n$ of $\mathcal H$ such
  that each $h_i$ is one of the following:
  \begin{equation}
    \begin{array}{lcr}
      h_i = x_i^2 + y_i^2 & \textup{for }  1 \leq i \leq k_e \ , &
\textup{(elliptic)}  \\
      h_i = x_iy_i  &  \textup{for }  k_e+1 \leq i \leq k_e+k_h \ , &
      \textup{(hyperbolic)}\\
      \begin{cases}
        h_i = x_i y_i + x_{i+1} y_{i+1} , \\
        h_{i+1} = x_i y_{i+1}-x_{i+1} y_i
      \end{cases} &
      \begin{array}{c}
        \textup{for }  i = k_e+k_h+ 2j-1,\\ 1\leq j \leq k_f
      \end{array}
      &
      \textup{(focus-focus pair)}
 \end{array}
  \end{equation}
\end{teo}

Observe that the number of elliptic components $k_e$, hyperbolic
components $k_h$ and focus-focus components $k_f$ is therefore an
invariant of the algebra $\mathcal H$. The triple $(k_e,k_h,k_f)$ is an invariant of the singularity and it is
called the Williamson type of $\mathcal H$. We have that
$n=k_e+k_h+2k_f$.  Let $h_1,\dots,h_n$ be a Williamson basis of this
Cartan subalgebra. We denote by $X_i$ the Hamiltonian vector field of
$h_i$ with respect to the Darboux form. These vector fields are a basis of the
corresponding Cartan subalgebra of $\mathfrak{sp}(2n,\CR)$. We say
that a vector field $X_i$ is hyperbolic (resp. elliptic) if the
corresponding function $h_i$ is so. We say that a pair of vector
fields $X_i,X_{i+1}$ is a focus-focus pair if $X_i$ and $X_{i+1}$ are
the Hamiltonian vector fields associated to functions $h_i$ and
$h_{i+1}$ in a focus-focus pair.

In the local coordinates specified above, the vector fields $X_i$ take
the following form:

\begin{itemize}
\item $X_i$ is an elliptic vector field, \begin{equation}X_i=
  2\left(-y_i\frac{\partial}{\partial x_i} +x_i{\frac{\partial}{\partial
      y_i}}\right) \ ;\end{equation}
\item $X_i$ is a hyperbolic vector field,
  \begin{equation}X_i=-x_i\frac{\partial}{\partial x_i}+y_i{\frac{\partial}{\partial
      y_i}} \ ;\end{equation}
\item $X_i,X_{i+1}$ is a focus-focus pair,
  \begin{equation}X_i=-x_i\frac{\partial}{\partial
    x_{i}}+y_{i}\frac{\partial}{\partial
    y_i}-x_{i+1}\frac{\partial}{\partial
    x_{i+1}}+y_{i+1}\frac{\partial}{\partial y_{i+1}}\end{equation}and
  \begin{equation}X_{i+1}=x_{i+1}\frac{\partial}{\partial
    x_i}+y_{i+1}\frac{\partial}{\partial
    y_i}-x_i\frac{\partial}{\partial
    x_{i+1}}-y_i\frac{\partial}{\partial y_{i+1}} \ .\end{equation}
\end{itemize}

Assume that $\mathcal{F}$ is a linear foliation on $\CR^{2n}$ with a rank 0 singularity at the
origin. Assume that the Williamson type of the singularity is
$(k_e,k_h,k_f)$. The linear model for the foliation is then generated by the
vector fields above, it turns out that these type of singularities are symplectically linearizable and we can read of the local symplectic geometry of the foliation from the algebraic data associated to the singularity (Williamson type).

This is the content of the following symplectic linearization result in \cite{eli1},\cite{eli2},\cite{Mi} (smooth category) and \cite{v} (analytic category),

\begin{teo}\label{thm:rank0uni}
Let $\omega$ be a smooth (resp. analytic) symplectic form defined in a neighbourhood $U$ of the origin and $\mathcal F$ a linear foliation with a rank zero singularity, of prescribed Williamson type, at the origin. Then, there exists a local diffeomorphism (resp. analytic diffeomorphism) $\phi:U\longrightarrow \phi(U)\subset\CR^{2n}$ such that $\phi$ preserves the foliation and $\phi^*(\somatorio{i=1}{n} \ud x_i\wedge \ud y_i)=\omega$, with $(x_1,y_1,\dots,x_n,y_n)$ local coordinates on $\phi(U)$.
\end{teo}

Futhermore, if $\mathcal{F}'$ is a foliation that has $\mathcal{F}$ as a linear foliation model near a point, one can symplectically linearize $\mathcal{F}'$ (see \cite{Mi}).

This is equivalent to Eliasson's theorem \cite{eli1, eli2} when the Williamson type of the singularity is $(k_e,0,0)$.

The classification of singularities of integrable system changes in the analytic category. This was already considered by Vey in \cite{v} and it is simpler because the Williamson type of the singularities is
$(k_e,k_h,0)$.

There are normal forms for higher rank which have been obtained by the first author together with Nguyen Tien Zung \cite{Mi, mirandazungequiv} also in the case of singular nondegenerate compact orbits. When the rank of the singularity is greater than 0, a collection of regular vector fields is also attached to it.


\section{A singular Poincar\'{e} lemma for a deformation complex}

This section revisits the main results contained in \cite{MiNg}.

Consider the family  $X_i$ of  singular vector fields given by Williamson's theorem above which form a
basis of a Cartan subalgebra of the Lie algebra $\mathfrak{sp}(2r,\CR)$ with $r\leq n$.

\begin{teo}[Miranda and Vu Ngoc]
  \label{theo:principal}
  Let $g_1,\dots g_r$, be a set of smooth functions on
  $\CR^{2n}$ with $r\leq n$ fulfilling the following commutation
  relations
  \begin{equation}X_i(g_j)=X_j(g_i), \quad \forall i,j \in\{1,\dots,r\}\end{equation}
  where the
  $X_i$'s are the vector fields defined above.  Then there exists a smooth function $G$ and $r$ smooth functions $f_i$
  such that,
\begin{equation} X_j(f_i)=0 \ ,  \ \forall  \ i,j \in\{1,\dots,r\} \ \ \text{and} \end{equation}
\begin{equation} g_i=f_i+X_i(G) \ , \ \forall \ i \in\{1,\dots,r\} \ . \end{equation}
\end{teo}

It is also included in \cite{MiNg} an interesting reinterpretation of this statement in terms of the deformation complex associated to an integrable system. We think that it is instructive to explain this succinctly here (we refer the reader to \cite{SanVN} and \cite{MiNg} for more details).

Using the same notation of the last section, let $\mathbf{h}=\langle h_1,\dots ,h_n\rangle_\CR$ and $\mathcal{C}_\mathbf{h}=\{f\in C^\infty(\CR^{2n}) \  ; \ X_h(f)=0, \ \forall \ h\in\mathbf{h} \}$. The set $\mathbf{h}$ is an Abelian Lie subalgebra of $(C^\infty(\CR^{2n}),\{\cdot,\cdot\})$ and $\mathcal{C}_\mathbf{h}$ is its centralizer.

The components of the moment map induce a representation of the commutative Lie algebra $\CR^n$ on $C^\infty(\CR^{2n})$,
\begin{equation}\CR^n\times C^\infty(\CR^{2n})\ni (v,f)\mapsto \{\mathbf{h}(v),f\}\in C^\infty(\CR^{2n}) \ .
\end{equation}Where, denoting by $(e_1,\dots,e_n)$ a basis of $\CR^n$, $v=v_1e_1+\cdots+v_ne_n$ and
\begin{equation}\{\mathbf{h}(v),f\}=v_1X_1(f)+\cdots+v_nX_n(f) \ .
\end{equation}

We can consider two Chevalley-Eilenberg complexes with the above action in mind, and the deformation complex is built from them. The first is the Chevalley-Eilenberg complex of $\CR^n$ with values in  $C^\infty(\CR^{2n})$, we denote $\mathrm{Hom}_\CR(\wedge^k\CR^n;C^\infty(\CR^{2n}))$ by $A^k$:
\begin{equation}
0\longrightarrow C^\infty(\CR^{2n})\longrightarrow A^1\longrightarrow A^2\longrightarrow A^3\longrightarrow\cdots \ .
\end{equation}The second is the Chevalley-Eilenberg complex of $\CR^n$ with values in  $C^\infty(\CR^{2n})/\mathcal{C}_\mathbf{h}$ (with respect to this action, $\CR^n$ acts trivially on $\mathcal{C}_\mathbf{h}$), where we denote $\mathrm{Hom}_\CR(\wedge^k\CR^n;C^\infty(\CR^{2n})/\mathcal{C}_\mathbf{h})$ by $B^k$:
\begin{equation}
0\longrightarrow C^\infty(\CR^{2n})/\mathcal{C}_\mathbf{h}\longrightarrow B^1\longrightarrow B^2\longrightarrow B^3\longrightarrow\cdots \ .
\end{equation}

Finally we define the \emph{deformation complex} as follows:
\begin{equation}
0\longrightarrow C^\infty(\CR^{2n})/\mathcal{C}_\mathbf{h}\stackrel{\bar{\ud}_{\mathbf{h}}}{\longrightarrow}B^1\stackrel{\partial_{\mathbf{h}}}{\longrightarrow}A^2\stackrel{\ud_{\mathbf{h}}}{\longrightarrow}A^3\stackrel{\ud_{\mathbf{h}}}{\longrightarrow}\cdots \ ,
\end{equation}the map $\partial_{\mathbf{h}}$ is defined by the following diagram (where all small triangles are commutative):
\[\xymatrix{ 0 \ar[r] &
  C^\infty(\CR^{2n}) \ar[r]^-{\ud_{\mathbf{h}}} \ar[d] &
 A^1 \ar[r]^-{\ud_{\mathbf{h}}} \ar[d] &
 A^2 \ar[r]^-{\ud_{\mathbf{h}}} \ar[d] &
  {\dots} \\
  0 \ar[r] &
 C^\infty(\CR^{2n})/\mathcal{C}_\mathbf{h} \ar[r]_-{\bar{\ud}_{\mathbf{h}}}
  \ar[ur]^{\partial_{\mathbf{h}}} &
  B^1 \ar[r]_-{\bar{\ud}_{\mathbf{h}}}
  \ar[ur]^{\partial_{\mathbf{h}}}&
  B^2 \ar[r]_-{\bar{\ud}_{\mathbf{h}}}
  \ar[ur]^{\partial_{\mathbf{h}}} &
  {\dots} }\]The cohomology groups associated to this complex are denoted by $H^k(\mathbf{h})$.

If $\alpha$ is a $1$-cocycle, then  for any smooth function $g_i$ with $\alpha(e_i)=[g_i]\in C^\infty(\CR^{2n})/\mathcal{C}_\mathbf{h}$ the commutation condition $X_i(g_j)=X_j(g_i)$ is fulfilled. Theorem \ref{theo:principal} says that there exists a function $G$ such that $g_i=f_i+ X_i(G)$, so $[g_i]=[X_i(G)]$ and this is exactly the coboundary condition.

Theorem \ref{theo:principal} combined with theorem \ref{thm:rank0uni} can be, then, reformulated as follows:

\begin{teo}[Miranda and Vu Ngoc]

 An integrable system with nondegenerate singularities
  is $C^\infty$-infinitesimally stable at the singular point,
  that is,
  \begin{equation}
  H^1(\mathbf{h})=0 .
  \end{equation}
\end{teo}

\section{Homotopy operators and a regular Poincaré lemma for foliated cohomology}\label{sec4homo}

Let us recall the following construction due to Guillemin and Sternberg \cite{geometricasymptotics} which generalizes, in a way \footnote{The proof contained in \cite{warner} makes a particular choice of retraction on star-shaped domains}, the classical proof of Poincaré lemma.

Consider $Y\subset M$ an embedded submanifold and let $\phi_t$ be a smooth retraction from $M$ to $Y$. Given any smooth  $k$-form $\alpha$, the following formula holds:
\begin{equation}
\alpha -\phi_0^*(\alpha)=\int_0^1 \dev{}{t}\phi_t^*(\alpha)=\int_0^1 \phi_t^*( \iota_{\xi_t}\ud\alpha) \ud t+ \ud\int_0^1 \phi_t^*( \iota_{\xi_t}\alpha) \ud t \ ,
\end{equation}where $\xi_t$ is the vector field associated to $\phi_t$.
Thus,  defining $I(\alpha)= \int_0^1 \phi_t^*( \iota_{\xi_t}\alpha) \ud t$,   we obtain,
\begin{equation}\label{eq:homotopy}
\alpha-\phi_0^*(\alpha)=I\circ\ud(\alpha)+ \ud\circ I(\alpha) \ .
\end{equation}

Now assume that $\alpha$ is a closed form, formula \ref{eq:homotopy} yields  $\alpha-\phi_0^*(\alpha)= \ud\circ I(\alpha)$, and therefore
$I(\alpha)$ is a primitive for the closed $k$-form $\alpha-\phi_0^*(\alpha)$.

This has been classically applied considering retractions to a point in contractible sets or to retractions to the base of a fiber bundle. In the context of Symplectic and Contact Geometry, this homotopy formula leads to the so-called Moser's path method \cite{Moser}.
As said before, formula \ref{eq:homotopy} does not, a priori, give a primitive for $\alpha$ but for the difference $ \alpha-\phi_0^*(\alpha)$\footnote{The vector field $\xi_t$ is the radial one when the retraction is $\phi_t(p_1,\dots,p_n)=(tp_1,\dots,tp_n)$, and this formula coincides with the one of Warner \cite{warner}, giving a primitive for $\alpha$.}.

This technique can also be applied for regular foliations. This approach using the general homotopy formula of Guillemin and Sternberg has the advantage that some choices on the retraction can be done in such a way that the vector field $\xi_t$ is tangent to special directions in $M$, thus, allowing an adaptation to the foliated cohomology case.


\subsection{Foliated cohomology}

Let $(M,\mathcal{F})$ be a foliated $m$-dimensional manifold and $n$ the dimension of the leaves. The (regular) foliation can be thought as a subbundle of $TM$, which is often denoted by $T\mathcal{F}$.

The foliated cohomology is the one associated to the following cochain complex:
\begin{equation}
0\longrightarrow C^\infty_\mathcal{F}(M)\hookrightarrow C^\infty(M)\stackrel{\ud_\mathcal{F}}{\longrightarrow}\Omega_\mathcal{F}^1(M)\stackrel{\ud_\mathcal{F}}{\longrightarrow}\cdots\stackrel{\ud_\mathcal{F}}{\longrightarrow}\Omega_\mathcal{F}^n(M)\stackrel{\ud_\mathcal{F}}{\longrightarrow}0 \ ,
\end{equation}where $\Omega_\mathcal{F}^k(M)=\Gamma(\wedge^kT\mathcal{F}^*)$, $C^\infty_\mathcal{F}(M)$ is the space of smooth functions which are constant along the leaves of the foliation, and $\ud_\mathcal{F}$ is the restriction of the exterior derivative, $\ud$, to $T\mathcal{F}$.\par

We can prove a Poincar\'{e} lemma for foliated cohomology, of a regular foliation, using equation \ref{eq:homotopy} by considering local coordinates in which the foliation is given by local equations $\ud p_{n+1}=0,\dots,\ud p_m=0$, and the retraction is given by $(tp_1,\dots ,tp_n,p_{n+1}\dots,p_m)$; the vector field $\xi_t$ is tangent to the relevant foliation.

\begin{teo}\label{poinca}
{\normalfont{[\textbf{Poincaré lemma for foliated cohomology}]}}
The foliated cohomology groups vanish for $\mathrm{degree}\geq 1$.
\end{teo}

One could try to mimic similar formulae to prove a singular Poincaré lemma for a foliation given by an integrable system with nondegenerate singularities. The main issue of adapting such a proof is the smoothness of the procedure. Indeed, as we will see later, the adaptation of such a procedure is not possible since the cohomology groups do not vanish if the foliation is singular.

Whilst the de Rham complex is a fine resolution for the constant sheaf $\CR$ on $M$, the foliated cohomology is a fine resolution for the sheaf of smooth functions which are constant along the leaves of the foliation.\par


\subsection{Geometric Quantization à la Kostant}

A symplectic manifold $(M,\omega)$ such that the de Rham class $[\omega]$ is integral is called prequantizable. A prequantum line bundle of $(M,\omega)$ is a Hermitian line bundle over $M$ with connection, compatible with the Hermitian structure, $(L,\nabla^\omega)$ that satisfies $curv(\nabla^\omega)=-i\omega$ (the curvature of $\nabla^\omega$ is proportional to the symplectic form). And a real polarization $\mathcal{F}$ is an integrable subbundle of $TM$ (the bundle $T\mathcal{F}$) whose leaves are Lagrangian submanifolds, i.e., $\mathcal{F}$ is a Lagrangian foliation.

The restriction of the connection $\nabla^{\omega}$ to the polarization induces an operator
\begin{equation}
\nabla:\Gamma(L)\to\Gamma(T\mathcal{F}^*)\otimes\Gamma (L) \ .
\end{equation}

Let $\mathcal{J}$ denotes the space of local sections $s$ of a prequantum line bundle $L$ such that $\nabla s=0$. The space $\mathcal{J}$ has the structure of a sheaf and it is called the sheaf of flat sections.

The quantization of $(M,\omega,L,\nabla,\mathcal{F})$ is given by
\begin{equation}
\mathcal{Q}(M)=\displaystyle\bigoplus_{k\geq 0}\check{H}^k(M;\mathcal{J}) \ ,
\end{equation}where $\check{H}^k(M;\mathcal{J})$ are \v{C}ech cohomology groups with values in the sheaf $\mathcal{J}$.

If $\mathcal{S}$ denotes the sheaf of sections of the line bundle $L$, the Kostant complex is
\begin{equation}\label{eq72}
0\longrightarrow \mathcal{J}\hookrightarrow \mathcal{S}\stackrel{\ud^\nabla}{\longrightarrow}\Omega_\mathcal{F}^1\otimes\mathcal{S}\stackrel{\ud^\nabla}{\longrightarrow}\cdots\stackrel{\ud^\nabla}{\longrightarrow}\Omega_\mathcal{F}^n\otimes\mathcal{S}\stackrel{\ud^\nabla}{\longrightarrow}0 \ ,
\end{equation}where $\ud^\nabla(\alpha\otimes s)=\ud_\mathcal{F}(\alpha)\otimes s+(-1)^{\mathrm{degree}(\alpha)}\alpha\wedge\nabla s$ and $\ud^\nabla\circ\ud^\nabla=0$ because the curvature of $\nabla$ vanishes along the leaves.

\begin{lem}\label{locflat} There is always a local unitary flat section on each point of $M$.
\end{lem}
\begin{proof} Let $U\subset M$ be a trivializing neighbourhood of $L$ with a unitary section $s:U\subset M\to L$. Since $\nabla s\in \Omega_{\mathcal{F}|_U}^1(U)\otimes\Gamma(L|_U)$ there is a $\alpha\in \Omega_{\mathcal{F}|_U}^1(U)$ such that $\nabla s=\alpha\otimes s$. The condition $\ud^\nabla\circ\ud^\nabla=0$ implies $\ud_\mathcal{F}\alpha=0$;
\begin{eqnarray}
0=\ud^\nabla(\nabla s)&=&\ud^\nabla(\alpha\otimes s)=\ud_\mathcal{F}\alpha\otimes s-\alpha\wedge\nabla s \nonumber \\
&=&\ud_\mathcal{F}\alpha\otimes s-(\alpha\wedge\alpha)\otimes s=\ud_\mathcal{F}\alpha\otimes s \ .
\end{eqnarray} By the Poincaré lemma for foliations (theorem \ref{poinca}) there exists a neighbourhood $V\subset U$ and $f\in C^\infty(V)$ such that $\ud_\mathcal{F}f=\alpha|_V$. Setting $r=\e^{-f}s|_V$,
\begin{equation}
\nabla r=\e^{-f}\nabla s|_V+\ud_\mathcal{F}(\e^{-f})\otimes s|_V=\e^{-f}(\alpha\otimes s)\big{|}_V-\e^{-f}\ud_\mathcal{F}f\otimes s|_V=0 \ ,
\end{equation}
so $r$ is a unitary flat section of $L|_V$.
\end{proof}		

Wherefore, for each point on $M$ there exists a trivializing neighbourhood $V\subset M$ of $L$ with a unitary flat section $s:V\subset M\to L$, and any element of $\Omega_\mathcal{F}^k(M)\otimes\Gamma(L)$ can be locally written as $\alpha\otimes s$, where $\alpha\in\Omega^k_{\mathcal{F}|_V}(V)$. The condition $\ud^\nabla(\alpha\otimes s)=0$ is, then, equivalent to $\ud_\mathcal{F}\alpha=0$, because $\ud^\nabla(\alpha\otimes s)=\ud_\mathcal{F}\alpha\otimes s+(-1)^k\alpha\wedge\nabla s$, $s\neq 0$ and $\nabla s=0$.

The Kostant complex is just the foliated complex twisted by the sheaf of sections $\mathcal{S}$, and exactness of the foliated complex implies exactness of the Kostant complex.

\begin{teo}\label{fineresolution} The Kostant complex is a fine resolution for $\mathcal{J}$. Therefore, its cohomology groups are isomorphic to the cohomology groups with coefficients in the sheaf of flat sections $\check{H}^k(M;\mathcal{J})$ and thus compute Geometric Quantization.
\end{teo}

Rawnsley provided a proof of this fact in \cite{JHR}.


\section{ The singular case}\label{sing}

The main objective of this section is to use the Poincar\'{e} lemma of the deformation complex to compute foliated cohomology.
We start this section by recalling a definition of foliated cohomology that is going to be used in the singular case. We then introduce some analytical tools that we need to compute these groups. These analytical tools are mainly a series of decomposition results for functions with respect to vector fields.
Finally, in the last subsection we enclose explicit computations of the cohomology groups.\par

Roughly, elements of these cohomology groups are given by a collection of functions wich are constant along the leaves of the foliation, fulfilling additional constraints. 


\subsection{Singular foliated cohomology}


Integrable systems defined on $(M,\omega)$ induce Lie subalgebras of $(\Gamma(TM),[\cdot,\cdot])$, namely

$(\mathcal{F}=\langle X_1,\dots,X_n \rangle_{C^\infty(M)}, [\cdot,\cdot]\big{|}_{\mathcal{F}})$, where $X_i$ is the Hamiltonian vector field of the $i$th component of a moment map $F:M\to\CR^n$. \par

Now, considering $C^\infty(M)$ as a $C^\infty(M)$-module, $(\mathcal{F}, [\cdot,\cdot]\big{|}_{\mathcal{F}})$ can be represented on $C^\infty(M)$ as vector fields acting on smooth functions.

This is an example of a Lie pseudo algebra representation (see \cite{Mkz} for precise definitions and a nice account for the history and, various, names of this structure) and one can, then, consider the following complex\footnote{The Lie pseudo algebra cohomology with respect to that particular representation.}:
\begin{equation}
0\longrightarrow C^\infty_\mathcal{F}(M)\hookrightarrow C^\infty(M)\stackrel{\ud_\mathcal{F}}{\longrightarrow}\Omega_\mathcal{F}^1(M)\stackrel{\ud_\mathcal{F}}{\longrightarrow}\cdots\stackrel{\ud_\mathcal{F}}{\longrightarrow}\Omega_\mathcal{F}^n(M)\stackrel{\ud_\mathcal{F}}{\longrightarrow}0 \ ,
\end{equation}With the differential defined by
\begin{eqnarray}
\ud_{\mathcal{F}}\alpha(Y_1,\dots,Y_{k+1})&=&\somatorio{i=1}{k+1}(-1)^{i+1}Y_i(\alpha(Y_1,\dots,\hat{Y_i},\dots,Y_{k+1}))  \nonumber \\
&&+\somatorio{i<j}{}(-1)^{i+j}\alpha([Y_i,Y_j],Y_1,\dots,\hat{Y_i},\dots,\hat{Y_j},\dots,Y_{k+1}) \ ,
\end{eqnarray}with $Y_1,\dots,Y_{k+1}\in\mathcal{F}$. The cochain spaces are defined by
\begin{equation}
\Omega_\mathcal{F}^k(M)=\mathrm{Hom}_{C^\infty(M)}(\wedge^k_{C^\infty(M)}\mathcal{F};C^\infty(M)) \ ,
\end{equation}and $C^\infty_\mathcal{F}(M)=\mathrm{ker}(\ud_{\mathcal{F}}:C^\infty(M)\to\Omega_\mathcal{F}^1(M))$.\par

The differential is a coboundary operator and the associated cohomology is denoted by $H^{\ \bullet}_{\mathcal{F}}(M)$.\par

\begin{rema}
This construction is also well defined in the analytic category and that is the notion used in theorem \ref{cohomologycomputed}.
\end{rema}

From now on $(M,\omega)$ will be a symplectic manifold near a rank zero nondegenerate singularity of Williamson type $(k_e,k_h,0)$. Thus $(\CR^{2n},\somatorio{i=1}{n}\ud x_i\wedge\ud y_i)$ is endowed with a distribution $\mathcal{F}$ generated by a Williamson basis.\par

\begin{defi}The vanishing set of a vector field of a Williamson basis $X_i$ is denoted by $\Sigma_i=\{p\in\CR^{2n} \ ; \ x_i(p)=y_i(p)=0\}$.
\end{defi}

\begin{prop}\label{welldefineform}If $\alpha\in\Omega_\mathcal{F}^k(\CR^{2n})$ then $\alpha(X_{j_1},\dots,X_{j_k})\big{|}_{\Sigma_{j_1}\cup\cdots\cup\Sigma_{j_k}}=0$.
\end{prop}
\begin{proof}At every point $p\in\CR^{2n}$ the map $\alpha\in\Omega_\mathcal{F}^k(\CR^{2n})$ reduces to an element of $\wedge^k{\mathcal{F}\big{|}_p}^{*}$. Since $X_i=0$ at $\Sigma_i$, for any $p\in\Sigma_i$ and vectors $Y_1(p),\dots,Y_{k-1}(p)\in\mathcal{F}\big{|}_p$, the following expression holds:
\begin{equation}
\alpha_p(X_i(p),Y_1(p),\dots,Y_{k-1}(p))=0 \ .
\end{equation}Therefore $\alpha(X_{j_1},\dots,X_{j_k})\big{|}_{\Sigma_i}=0$ for $i=j_1,\dots,j_k$.
\end{proof}


\subsection{Analytical tools: Special decomposition of smooth functions}

Here we present special decompositions for functions with respect to vector fields of a Williamson basis. In order to fix notation, we recall what we mean by a \emph{flat function at a subset},

\begin{defi}Consider $\CR^m$ endowed with coordinates $(p_1,\dots,p_m)$. A smooth function $g\in C^\infty(\CR^m)$ is said to be Taylor flat at the subset $\{p_1=\cdots=p_k=0\}$ when
\begin{equation}
\fraction{\partial^{j_1+\cdots +j_k}g}{\partial p_1^{j_1}\cdots\partial p_k^{j_k}}\Bigg{|}_{\{p_1=\cdots=p_k=0\}}=0 \ ,
\end{equation}for all $j_1,\dots,j_k$ and some fixed $k\leq m$.
\end{defi}

\begin{rema}
It is important to point out that if $g\in C^\omega(\CR^m)$ is Taylor flat at $\{p_1=\cdots=p_k=0\}$, then it is the zero function.
\end{rema}

\begin{defi}Given integer numbers $k_e$ and $k_h$, we will call a smooth function $f\in C^\infty(\CR^{2n})$ complanate if it can be written as
\begin{equation}
f=\somatorio{i=k_e+1}{k_e+k_h}T_i \ ,
\end{equation}where each $T_i$ is a Taylor flat function at $\Sigma_i$. A noncomplanate function is one for which such an expression cannot be found.
\end{defi}

We can find special decompositions for smooth functions like $f=f_i+X_i(F_i)$. The following result is a summary of results contained in \cite{Mi} and \cite{MiNg},
\begin{lem}\label{Evalemma}Assume that the origin is a singularity of Williamson type $(k_e,k_h,0)$, then for any $f\in C^\infty(\CR^{2n})$ there exist $f_i,F_i\in C^\infty(\CR^{2n})$ such that, for each vector field $X_i$ in a Williamson basis, $f=f_i+X_i(F_i)$. Moreover,
\begin{enumerate}
  \item $X_i(f_i)=0$;
  \item $f_i$ is uniquely defined if $X_i$ defines an $S^1$-action, if not $f_i$ is uniquely defined up to Taylor flat functions at $\Sigma_i$;
  \item one can choose $f_i$ and $F_i$ such that $X_j(f_i)=X_j(F_i)=0$ whenever $X_j(f)=0$ for $j\neq i$;
	\item if $f$ vanishes at the zero set of any vector of a Williamson basis, so does the function $f_i$ and one can choose $F_i$ vanishing at the zero set, as well;
	\item $X_i(f)=0$ implies that $f$ depends on $x_i$ and $y_i$ via $h_i$:
\begin{equation*}
f(x_1,y_1,\dots,x_n,y_n)=\tilde{f}(x_1,y_1,\dots,x_i^2+y_i^2,\dots,x_n,y_n)\ ,
\end{equation*}
\begin{equation*}
f(x_1,y_1,\dots,x_n,y_n)\big{|}_{Q_i^j}=\tilde{f}(x_1,y_1,\dots,x_iy_i,\dots,x_n,y_n)\ ,
\end{equation*}where $Q_i^1=\{x_i>0,y_i>0\}$, $Q_i^2=\{x_i>0,y_i<0\}$, $Q_i^3=\{x_i<0,y_i>0\}$ and $Q_i^4=\{x_i<0,y_i<0\}$.
\end{enumerate}
\end{lem}

The case when $X_i$ is an elliptic vector field was proved in \cite{eli1,eli2,Mi}; \cite{Mi} also has a proof when $X_i$ is a hyperbolic vector field.
\begin{rema}
The proofs contained in \cite{Mi} can be adapted for the analytic category: for hyperbolic singularities the formal proof yields the corresponding analytic statement, whilst the integrals defining the elliptic decomposition entail the analyticity of the construction. This is the version of the lemma used in the proof of theorem \ref{cohomologycomputed}. Furthermore, the uniqueness of the decomposition holds for both types of singularities, since there are no flat functions (apart from the zero function) in the analytic category.
\end{rema}


\subsection{Computation of foliated cohomology groups}

We will distinguish between the smooth and the analytic category.

In the smooth case we can completely determine the cohomology groups in degree $1$ and $n$ for Williamson type $(k_e,k_h,0)$, and in all degrees for Williamson type $(k_e,0,0)$. In the analytic case the computations are done in all degrees.

\begin{teo}\label{smoothcohomologycomputed1}
{\normalfont{[\textbf{Degree 1 smooth case}]}} Consider $(\CR^{2n},\somatorio{i=1}{n}\ud x_i\wedge\ud y_i)$ endowed with a smooth distribution $\mathcal{F}$ generated by a Williamson basis of type $(k_e,k_h,0)$, then the following decomposition holds:
\begin{equation}
\mathrm{ker}(\ud_\mathcal{F}:\Omega_\mathcal{F}^1(\CR^{2n})\to\Omega_\mathcal{F}^2(\CR^{2n}))=W_\mathcal{F}^1(\CR^{2n})\oplus\ud_\mathcal{F}(C^\infty(\CR^{2n})) \ ,
\end{equation}where $W_\mathcal{F}^1(\CR^{2n})$ is the set of $1$-forms $\beta\in\Omega_\mathcal{F}^1(\CR^{2n})$ such that $\pounds_{X_i}(\beta)=0$ for all $i$, and if $X_i$ is of hyperbolic type $\beta(X_i)$ is not Taylor flat at $\Sigma_i$ (when it is nonzero).

Thus, the foliated cohomology group in degree $1$ is given by:
\begin{eqnarray*}
H_\mathcal{F}^1(\CR^{2n})&\cong&\displaystyle\bigoplus_{i=1}^{k_e}\{f_i\in C_\mathcal{F}^\infty(\CR^{2n}) \ ; \  f\big{|}_{\Sigma_i}=0 \} \\
&&\displaystyle\bigoplus_{i=k_e+1}^{n}\{f_i\in C_\mathcal{F}^\infty(\CR^{2n}) \ ; \ f=0 \ or \ f\big{|}_{\Sigma_i}=0 \ \text{and not Taylor flat at $\Sigma_i$} \}
\end{eqnarray*}
\end{teo}
\begin{proof}For any $\alpha\in\Omega_\mathcal{F}^1(\CR^{2n})$ the condition $\ud_\mathcal{F}\alpha=0$ implies
\begin{equation}
\ud_\mathcal{F}\alpha(X_i,X_j)=X_i(\alpha(X_j))-X_j(\alpha(X_i))=0 \ ,
\end{equation}and theorem \ref{theo:principal} says that $\alpha(X_i)=f_i+X_i(F)$, where $F\in C^\infty(\CR^{2n})$ and $f_i\in C^\infty_\mathcal{F}(\CR^{2n})$. Thus any closed foliated $1$-form $\alpha$ is cohomologous to a foliated $1$-form $\beta$ satisfying $\pounds_{X_i}(\beta)=0$ for all $i$ (proposition \ref{welldefineform} and item 4 of lemma \ref{Evalemma} guarantee that the forms are well defined); the condition $\pounds_{X_i}(\beta)=0$ automatically implies that $\beta$ is closed.\par

There exists $g\in C^\infty(\CR^{2n})$ such that $\ud_\mathcal{F}g=\beta$ if and only if $\beta(X_i)=X_i(g)$. Since $\pounds_{X_i}(\beta)=0$, this implies $X_i(\beta(X_i))=0$
and by uniqueness (up to Taylor flat functions, lemma \ref{Evalemma}) $0=\beta(X_i)+X_i(-g)$ has a solution if and only if $\beta(X_i)=0$ or $\beta(X_i)$ is Taylor flat at $\Sigma_i$ (for $i=k_e+1,\dots,n$). Wherefore, $\beta$ is exact if and only if $\beta=0$ or, if $\beta(X_i)\neq 0$ (for $i=k_e+1,\dots,n$), $\beta(X_i)$ is Taylor flat at $\Sigma_i$.\par

The expression $\mathrm{ker}=W_\mathcal{F}^1(\CR^{2n})\oplus\ud_\mathcal{F}(C^\infty(\CR^{2n}))$ implies $H_\mathcal{F}^1(\CR^{2n})=W_\mathcal{F}^1(\CR^{2n})$, by definition any $\beta\in W_\mathcal{F}^1(\CR^{2n})$ can be given by $n$ functions vanishing at certain points (proposition \ref{welldefineform}) and satisfying some Taylor flat condition, e.g.: $\beta(X_n)=f\in C^\infty(\CR^{2n})$, $f\big{|}_{\Sigma_i}=0$ and not Taylor flat at $\Sigma_n$, if it is nonzero. The Lie derivative condition yields $f\in C^\infty_\mathcal{F}(\CR^{2n})$.
\end{proof}

We now consider the case of top degree forms in the smooth category,
\begin{teo}\label{smoothcohomologycomputedn}
{\normalfont{[\textbf{Top degree smooth case}]}} Consider $(\CR^{2n},\somatorio{i=1}{n}\ud x_i\wedge\ud y_i)$ endowed with a smooth distribution $\mathcal{F}$ generated by a Williamson basis of type $(k_e,k_h,0)$, then the following decomposition holds:
\begin{equation}
\Omega_\mathcal{F}^n(\CR^{2n})=W_\mathcal{F}^n(\CR^{2n})\oplus\ud_\mathcal{F}(\Omega_\mathcal{F}^{n-1}(\CR^{2n})) \ ,
\end{equation}where $W_\mathcal{F}^n(\CR^{2n})$ is the set of $n$-forms $\beta\in\Omega_\mathcal{F}^n(\CR^{2n})$ such that $\pounds_{X_i}(\beta)=0$ for all $i$, and if $\beta(X_1,\dots,X_n)\neq 0$, it is noncomplanate.

Thus, the foliated cohomology group in degree $n$ is given by:
\begin{eqnarray*}
H_\mathcal{F}^n(\CR^{2n})&\cong&\displaystyle\{f\in C_\mathcal{F}^\infty(\CR^{2n}) \\
&;& f\big{|}_{\Sigma_1\cup\cdots\cup\Sigma_n}=0 \ \text{and $f$ is noncomplanate or zero}\}
\end{eqnarray*}
\end{teo}
\begin{proof}For $\alpha\in\Omega_\mathcal{F}^n(\CR^{2n})$ it holds $\ud_\mathcal{F}\alpha=0$. Since $\alpha(X_1,\dots,X_n)\in C^\infty(\CR^{2n})$, lemma \ref{Evalemma} asserts that $\alpha(X_1,\dots,X_n)=f_1+X_1(F_{2\cdots n})$ with $X_1(f_1)=0$. Applying again lemma \ref{Evalemma}, $f_1=f_2+X_2(-F_{13\cdots n})$ with $X_2(f_2)=0$, but also $X_1(f_2)=0$ because $X_1(f_1)=0$. Repeating this process for all $X_i$, one finally gets
\begin{equation}
\alpha(X_1,\dots,X_n)=f+\somatorio{i=1}{n}(-1)^{i+1}X_i(F_{1\cdots\hat{i}\cdots n}) \ ,
\end{equation}with $f\in C^\infty_\mathcal{F}(\CR^{2n})$, i.e.: there exists $\beta\in\Omega_\mathcal{F}^n(\CR^{2n})$ and $\zeta\in\Omega_\mathcal{F}^{n-1}(\CR^{2n})$ satisfying $\alpha=\beta+\ud_\mathcal{F}\zeta$ and $\pounds_{X_i}(\beta)=0$ for all $i$ (again, proposition \ref{welldefineform} and item 4 of lemma \ref{Evalemma} guarantee that the forms are well defined).\par

The foliated $n$-form $\beta$ is exact if and only if there exists $\sigma\in\Omega_\mathcal{F}^{n-1}(\CR^{2n})$ such that
\begin{equation}\label{betaeqn}
\beta(X_1,\dots,X_n)=\somatorio{i=1}{n}(-1)^{i+1}X_i(\sigma(X_1,\dots,\hat{X_i},\dots,X_n)) \ .
\end{equation}Applying lemma \ref{Evalemma},
\begin{equation}\label{sigmaeqn}
\sigma(X_1,\dots,\hat{X_i},\dots,X_n)=g_{1\cdots\hat{i}\cdots n}^1+X_1(G_{1\cdots\hat{i}\cdots n}^1) \ ,
\end{equation}with $X_1(g_{1\cdots\hat{i}\cdots n}^1)=0$. Then, substituting equation \ref{sigmaeqn} in \ref{betaeqn}, using $[X_i,X_j]=0$ and invoking lemma \ref{Evalemma},
\begin{equation}
0=\beta(X_1,\dots,X_n)+\somatorio{i=2}{n}(-1)^{i}X_i(g_{1\cdots\hat{i}\cdots n}^1)+X_1\left(\somatorio{i=1}{n}(-1)^{i}X_i(G_{1\cdots\hat{i}\cdots n}^1)\right)
\end{equation}has solution if and only if
\begin{equation}\label{betaT1eqn}
T_1=\beta(X_1,\dots,X_n)+\somatorio{i=2}{n}(-1)^{i}X_i(g_{1\cdots\hat{i}\cdots n}^1) \ ,
\end{equation}where $X_1(T_1)=0$ and, if $T_1\neq 0$, it is Taylor flat at $\Sigma_1$.\par

Once more, applying lemma \ref{Evalemma} to $T_1$ and $g_{1\cdots\hat{i}\cdots n}^1$ with respect to $X_2$,
\begin{equation}\label{T1}
T_1=T_{12}+X_2(t_{12})
\end{equation}and
\begin{equation}\label{g1}
g_{1\cdots\hat{i}\cdots n}^1=g_{1\cdots\hat{i}\cdots n}^{12}+X_2(G_{1\cdots\hat{i}\cdots n}^{12}) \ ,
\end{equation}where $X_2(T_{12})=X_2(g_{1\cdots\hat{i}\cdots n}^{12})=0$, $X_1(T_{12})=X_1(g_{1\cdots\hat{i}\cdots n}^{12})=0$ and $T_{12}$ is Taylor flat at $\Sigma_1$ because $X_1(T_1)=X_1(g_{1\cdots\hat{i}\cdots n}^{12})=0$ and $T_1$ is Taylor flat at $\Sigma_1$.\par

Now, replacing equations \ref{T1} and \ref{g1} in \ref{betaT1eqn}, using $[X_i,X_j]=0$ and because of lemma \ref{Evalemma},
\begin{equation}
0=\beta(X_1,\dots,X_n)-T_{12}+\somatorio{i=3}{n}(-1)^{i}X_i(g_{1\cdots\hat{i}\cdots n}^{12})+X_2\left(-t_{12}+\somatorio{i=2}{n}(-1)^{i}X_i(G_{1\cdots\hat{i}\cdots n}^{12})\right)
\end{equation}has solution if and only if
\begin{equation}
T_2+T_{12}=\beta(X_1,\dots,X_n)+\somatorio{i=3}{n}(-1)^{i}X_i(g_{1\cdots\hat{i}\cdots n}^{12}) \ ,
\end{equation}where $X_1(T_2)=X_2(T_2)=0$ and, if $T_2\neq 0$, it is Taylor flat at $\Sigma_2$.

The next step is to decompose $T_2$, $T_{12}$ and $g_{1\cdots\hat{i}\cdots n}^{12}$ with respect to $X_3$ and argue as before. Continuing with this process for all $X_i$ one obtains $\beta(X_1,\dots,X_n)=T_n+T_{(n-1)n}+\cdots+T_{1\cdots n}$, where $T_n,\dots,T_{1\cdots n}\in C^\infty_\mathcal{F}(\CR^{2n})$ and, if $T_{i\cdots n}\neq 0$, it is Taylor flat at $\Sigma_i$.

We were assuming $k_e=0$ and $k_h=n$. The case when $k_e\neq 0$ is straightforward: just forget about Taylor flatness for those indices.

From $\Omega_\mathcal{F}^n(\CR^{2n})=W_\mathcal{F}^n(\CR^{2n})\oplus\ud_\mathcal{F}(\Omega_\mathcal{F}^{n-1}(\CR^{2n}))$ we obtain $H_\mathcal{F}^n(\CR^{2n})=W_\mathcal{F}^n(\CR^{2n})$, by definition any $\beta\in W_\mathcal{F}^n(\CR^{2n})$ can be given by a function vanishing at certain points (proposition \ref{welldefineform}) and being noncomplanate: $\beta(X_1,\dots,X_n)=f\in C^\infty(\CR^{2n})$, $f=0$ at $\Sigma_1\cup\cdots\cup\Sigma_n$ and is noncomplanate, if nonzero. The Lie derivative condition further implies that such a function is constant along the leaves.
\end{proof}

\begin{rema}The proofs of theorems \ref{smoothcohomologycomputed1} and \ref{smoothcohomologycomputedn} also hold in the analytic category after, obvious and minor, modifications (essentially getting read of Taylor flat functions).
\end{rema}


Before proving theorem \ref{cohomologycomputed}, it is worthwhile to look at a particular (smooth) case to illustrate its intricacy.

\begin{prop}\label{e/h-ff} Consider $(\CR^6,\somatorio{i=1}{3}\ud x_i\wedge\ud y_i)$ with $h_1,h_2,h_3\in C^\infty(\CR^6)$ a Williamson basis. If both $X_1,X_2$ are of hyperbolic type and $X_3$ is of elliptic type, then:
\begin{equation}
\mathrm{ker}(\ud_\mathcal{F}:\Omega_\mathcal{F}^2(\CR^6)\to\Omega_\mathcal{F}^3(\CR^6))=W_\mathcal{F}^2(\CR^6)\oplus\ud_\mathcal{F}(\Omega_\mathcal{F}^1(\CR^6)) \ ,
\end{equation}where $W_\mathcal{F}^2(\CR^6)$ is the set of $2$-forms $\beta\in\Omega_\mathcal{F}^2(\CR^6)$ such that $\pounds_{X_i}(\beta)=0$ for $i=1,2,3$, and if $\beta(X_i,X_j)\neq 0$ it is noncomplanate.\par
\end{prop}
\begin{proof}The condition $\ud_\mathcal{F}\alpha=0$ implies, for any $\alpha\in\Omega_\mathcal{F}^2(\CR^6)$,
\begin{equation}
0=X_1(\alpha(X_2,X_3))-X_2(\alpha(X_1,X_3))+X_3(\alpha(X_1,X_2)) \ .
\end{equation}Lemma \ref{Evalemma} gives
\begin{equation}\label{eq23}
\alpha(X_1,X_3)=f_{13}+X_3(F_{13}) \ \text{and} \  \alpha(X_2,X_3)=f_{23}+X_3(F_{23}) \ ,
\end{equation}with $X_3(f_{13})=X_3(f_{23})=0$.\par

Because $[X_i,X_j]=0$,
\begin{equation}
0=X_1(f_{23})-X_2(f_{13})+X_3\left(\alpha(X_1,X_2)+X_1(F_{23})-X_2(F_{13})\right) \ ,
\end{equation}by uniqueness (lemma \ref{Evalemma}),
\begin{equation}\label{eq12}
\alpha(X_1,X_2)=f_{12}+X_2(F_{13})-X_1(F_{23}) \ ,
\end{equation}with $X_3(f_{12})=0$ and
\begin{equation}
X_1(f_{23})=X_2(f_{13}) \ .
\end{equation}\par

Defining $\alpha_3\in\Omega_\mathcal{F}^1(\CR^6)$ by
\begin{equation}
\alpha_3(X_1)=f_{13} \ , \ \alpha_3(X_2)=f_{23} \ \text{and} \ \alpha_3(X_3)=0 \ ,
\end{equation}it is clear that $\ud_\mathcal{F}\alpha_3=0$ (proposition \ref{welldefineform} and item 4 of lemma \ref{Evalemma} guarantee that it is well defined). Theorem \ref{smoothcohomologycomputed1}, then, implies $\alpha_3=\beta_3+\ud_\mathcal{F}G_3$, with $\beta_3\in W_\mathcal{F}^1(\CR^6)$. In other words:
\begin{equation}\label{eqi13}
f_{13}=g_{13}+X_1(G_3) \ , \ f_{23}=g_{23}+X_2(G_3) \ \text{and} \ X_3(G_3)=0
\end{equation}\par

Applying repeatedly lemma \ref{Evalemma}, for each $X_i$ with $i\neq 3$, to the function $f_{12}$ one gets
\begin{equation}\label{eqi12}
f_{12}=g_{12}-X_1(G_{23})+X_2(G_{13}) \ ,
\end{equation}with $X_1(g_{12})=X_2(g_{12})=0$ and $X_3(g_{12})=X_3(G_{13})=X_3(G_{23})=0$, because $X_3(f_{12})=0$.\par

Summing up, plugging equation \ref{eqi13} in equation \ref{eq23}, using $X_3(G_{13})=X_3(G_{23})=0$, and equation \ref{eqi12} in equation \ref{eq12}:
\begin{eqnarray}\label{eqdeco13}
\alpha(X_1,X_3)&=&g_{13}+X_1(G_3)+X_3(F_{13}+G_{13}) \\
\alpha(X_2,X_3)&=&g_{23}+X_2(G_3)+X_3(F_{23}+G_{23})
\end{eqnarray}and
\begin{equation}\label{eqdeco23}
\alpha(X_1,X_2)=g_{12}-X_1(F_{23}+G_{23})+X_2(F_{13}+G_{13}) \ .
\end{equation}Wherefore $\alpha=\beta+\ud_\mathcal{F}\zeta$ with $\beta\in W_\mathcal{F}^1(\CR^6)$;
\begin{equation}
\beta(X_1,X_2)=g_{12} \ , \ \beta(X_1,X_3)=g_{13} \ , \ \beta(X_2,X_3)=g_{23}
\end{equation}and
\begin{equation}
\zeta(X_1)=-F_{13}-G_{13} \ , \ \zeta(X_2)=-F_{23}-G_{23} \ , \ \zeta(X_3)=G_3
\end{equation}(as always, proposition \ref{welldefineform} and item 4 of lemma \ref{Evalemma} guarantee that the forms are well defined).\par

The condition $\pounds_{X_i}(\beta)=0$ for $i=1,2,3$ implies $\ud_\mathcal{F}\beta=0$, and there exists $\sigma\in\Omega_\mathcal{F}^1(\CR^6)$ such that $\ud_\mathcal{F}\sigma=\beta$ if and only if
\begin{equation}\label{betaeq13}
\beta(X_i,X_j)=X_i(\sigma(X_j))-X_j(\sigma(X_i)) \ .
\end{equation}\par

Applying lemma \ref{Evalemma},
\begin{equation}\label{sigmaeq13}
\sigma(X_i)=s_{i3}+X_3(S_{i3}) \ ,
\end{equation}with $X_3(s_{i3})=0$. Then, plugging equation \ref{sigmaeq13} in \ref{betaeq13}, using $[X_i,X_j]=0$ and using uniqueness (lemma \ref{Evalemma}),
\begin{equation}
0=\beta(X_i,X_j)+X_j(s_{i3})-X_i(s_{j3})+X_3(X_j(S_{i3})-X_i(S_{j3}))
\end{equation}has solution if and only if
\begin{equation}\label{betaeq23}
0=\beta(X_i,X_j)+X_j(s_{i3})-X_i(s_{j3}) \ .
\end{equation}\par

Again, applying lemma \ref{Evalemma},
\begin{equation}\label{sigmaeq23}
s_{i3}=s_{i23}+X_2(S_{i23}) \ ,
\end{equation}with $X_2(s_{i23})=0$ and $X_3(s_{i23})=0$, because $X_3(s_{i3})=0$. Then, replacing equation \ref{sigmaeq23} in \ref{betaeq23}, using $[X_i,X_j]=0$ and because of lemma \ref{Evalemma},
\begin{equation}
0=\beta(X_i,X_j)+X_j(s_{i23})-X_i(s_{j23})+X_2(X_j(S_{i23})-X_i(S_{j23}))
\end{equation}has solution if and only if
\begin{equation}
T_{ij2}=\beta(X_i,X_j)+X_j(s_{i23})-X_i(s_{j23}) \ ,
\end{equation}where $X_3(T_{ij2})=X_2(T_{ij2})=0$ and $T_{ij2}$ is Taylor flat at $\Sigma_2$. Explicitly,
\begin{equation}\label{betaeq23Ta}
T_{122}=\beta(X_1,X_2)-X_1(s_{223}) \ ,
\end{equation}
\begin{equation}\label{betaeq23Tb}
T_{132}=\beta(X_1,X_3)-X_1(s_{323})
\end{equation}and
\begin{equation}\label{betaeq23Tc}
T_{232}=\beta(X_2,X_3) \ .
\end{equation}\par

Once more, applying lemma \ref{Evalemma},
\begin{equation}\label{sigmaeq12}
T_{ij2}=t_{ij}+X_1(T_{ij12}) \ ,
\end{equation}with $X_1(t_{ij})=0$, $X_2(t_{ij})=X_3(t_{ij})=0$ and $t_{ij}$ is Taylor flat at $\Sigma_2$, because $X_2(T_{ij2})=X_3(T_{ij2})=0$ and $T_{ij2}$ is Taylor flat at $\Sigma_2$. Then, substituting equation \ref{sigmaeq12} in \ref{betaeq23Ta}, \ref{betaeq23Tb} and \ref{betaeq23Tc}, using $[X_i,X_j]=0$ and using lemma \ref{Evalemma},
\begin{equation}
\beta(X_1,X_2)-t_{12}=X_1(s_{223}+T_{1212}) \ ,
\end{equation}
\begin{equation}
\beta(X_1,X_3)-t_{13}=X_1(s_{323}+T_{1312})
\end{equation}and
\begin{equation}
\beta(X_2,X_3)-t_{23}=X_1(T_{2312}) \ .
\end{equation}have solution if and only if
\begin{equation}
\beta(X_1,X_2)=t_{12}+T_{12} \ , \ \beta(X_1,X_3)=t_{13}+T_{13} \ \text{and} \ \beta(X_2,X_3)=t_{23}+T_{23}
\end{equation}where each $T_{ij}\in C^\infty_\mathcal{F}(\CR^6)$ and is Taylor flat at $\Sigma_1$.
\end{proof}


\begin{teo}\label{cohomologycomputed}
{\normalfont{[\textbf{Analytic case}]}} Consider $(\CR^{2n},\somatorio{i=1}{n}\ud x_i\wedge\ud y_i)$ endowed with a analytic distribution $\mathcal{F}$ generated by a Williamson basis of type $(k_e,k_h,0)$, then the following decomposition holds:
\begin{equation}
\mathrm{ker}(\ud_\mathcal{F}:\Omega_\mathcal{F}^k(\CR^{2n})\to\Omega_\mathcal{F}^{k+1}(\CR^{2n}))=W_\mathcal{F}^k(\CR^{2n})\oplus\ud_\mathcal{F}(\Omega_\mathcal{F}^{k-1}(\CR^{2n})) \ ,
\end{equation}where $W_\mathcal{F}^k(\CR^{2n})$ is the set of analytic $k$-forms $\beta\in\Omega_\mathcal{F}^k(\CR^{2n})$ such that $\pounds_{X_i}(\beta)=0$ for all $i$.
\end{teo}
\begin{proof}It remains to prove when the degree is different from $1$ and $n$, since the proofs of theorems \ref{smoothcohomologycomputed1} and \ref{smoothcohomologycomputedn}, as mentioned, work for these particular cases.\par

If $\alpha\in\Omega_\mathcal{F}^k(\CR^{2n})$ the condition $\ud_\mathcal{F}\alpha=0$ implies
\begin{equation}\label{eq00}
0=X_{j_1}(\alpha(X_{j_2},\dots,X_{j_{k+1}}))+\somatorio{i=2}{k+1}(-1)^{i+1}X_{j_i}(\alpha(X_{j_1},\dots,\hat{X_{j_i}},\dots,X_{j_{k+1}})) \ .
\end{equation}Applying successively of lemma \ref{Evalemma}, with respect to each $X_{j_i}$ ($i\neq 1$), gives
\begin{eqnarray}\label{eq1k+1}
\alpha(X_{j_1},\dots,\hat{X_{j_i}},\dots,X_{j_{k+1}})&=&g_{j_1\cdots\hat{j_i}\cdots j_{k+1}}^{j_1\cdots j_{k+1}}+X_{j_1}(F_{j_1\cdots\hat{j_i}\cdots j_{k+1}}^{j_1\cdots j_{k+1}}) \nonumber \\ & & +\somatorio{\underset{l\neq i}{l=2}}{k+1}(-1)^{l_i+1}X_{j_l}(G_{j_1j_2\cdots\hat{j_i}\cdots\hat{j_l}\cdots j_{k+1}}^{j_1\cdots j_{k+1}}) \ ,
\end{eqnarray}where $X_{j_m}(g_{j_1\cdots\hat{j_i}\cdots j_{k+1}}^{j_1\cdots j_{k+1}})=0$ for $m=1,\dots,k+1\neq i$ and
\begin{equation}
l_i=\left\{\begin{array}{l}
l \ \text{if} \ l<i \\ 
l+1 \ \text{if} \ l>i
\end{array}\right. \ .
\end{equation}\par


Substituting equation \ref{eq1k+1} in \ref{eq00} (using $[X_i,X_j]=0$),
\begin{eqnarray}
0&=&\somatorio{i=2}{k+1}(-1)^{i+1}X_{j_i}\left(g_{j_1\cdots\hat{j_i}\cdots j_{k+1}}^{j_1\cdots j_{k+1}}+\somatorio{\underset{l\neq i}{l=2}}{k+1}(-1)^{l_i+1}X_{j_l}(G_{j_1j_2\cdots\hat{j_i}\cdots\hat{j_l}\cdots j_{k+1}}^{j_1\cdots j_{k+1}})\right) \nonumber \\
&&+X_{j_1}\left(\alpha(X_{j_2},\dots,X_{j_{k+1}})+\somatorio{i=2}{k+1}(-1)^{i+1}X_{j_i}(F_{j_1\cdots\hat{j_i}\dots j_{k+1}}^{j_1\cdots j_{k+1}})\right) \ ,
\end{eqnarray}and by uniqueness (lemma \ref{Evalemma}):
\begin{equation}\label{eq2k+1}
\alpha(X_{j_2},\dots,X_{j_{k+1}})=f_{j_2\cdots j_{k+1}}^{j_1\cdots j_{k+1}}+\somatorio{i=2}{k+1}(-1)^{i}X_{j_i}(F_{j_1\cdots\hat{j_i}\cdots j_{k+1}}^{j_1\cdots j_{k+1}}) \ ,
\end{equation}with $X_{j_1}(f_{j_2\cdots j_{k+1}}^{j_1\cdots j_{k+1}})=0$. 


Again, applying repeatedly lemma \ref{Evalemma}, for each $X_{j_i}$ with $i\neq 1$, to the function $f_{j_2\cdots j_{k+1}}^{j_1\cdots j_{k+1}}$ one gets
\begin{equation}\label{eqi2k+1}
f_{j_2\cdots j_{k+1}}^{j_1\cdots j_{k+1}}=g_{j_2\cdots j_{k+1}}^{j_1\cdots j_{k+1}}+\somatorio{i=2}{k+1}(-1)^{i}X_{j_i}(G_{j_1\cdots\hat{j_i}\cdots j_{k+1}}^{j_1\cdots j_{k+1}}) \ ,
\end{equation}with $X_{j_i}(g_{j_2\cdots j_{k+1}}^{j_1\cdots j_{k+1}})=0$ for $i=1,\dots,k+1$ and $X_{j_1}(G_{j_1\cdots\hat{j_i}\cdots j_{k+1}}^{j_1\cdots j_{k+1}})=0$, because $X_{j_1}(f_{j_2\cdots j_{k+1}}^{j_1\cdots j_{k+1}})=0$.\par

Using $X_{j_1}(G_{j_1\cdots\hat{j_i}\cdots j_{k+1}}^{j_1\cdots j_{k+1}})=0$, and substituting equation \ref{eqi2k+1} in equation \ref{eq2k+1}:
\begin{eqnarray}\label{eqdeco1k+1}
\alpha(X_{j_1},\dots,\hat{X_{j_i}},\dots,X_{j_{k+1}})&=&g_{j_1\cdots\hat{j_i}\cdots j_{k+1}}^{j_1\cdots j_{k+1}}+X_{j_1}(F_{j_1\cdots\hat{j_i}\cdots j_{k+1}}^{j_1\cdots j_{k+1}}+G_{j_1\cdots\hat{j_i}\cdots j_{k+1}}^{j_1\cdots j_{k+1}}) \nonumber \\ & & +\somatorio{\underset{l\neq i}{l=2}}{k+1}(-1)^{l_i+1}X_{j_l}(G_{j_1j_2\cdots\hat{j_i}\cdots\hat{j_l}\cdots j_{k+1}}^{j_1\cdots j_{k+1}}) \ ,
\end{eqnarray}for $i\neq 1$, and
\begin{equation}\label{eqdeco2k+1}
\alpha(X_{j_2},\dots,X_{j_{k+1}})=g_{j_2\cdots j_{k+1}}^{j_1\cdots j_{k+1}}+\somatorio{i=2}{k+1}(-1)^{i}X_{j_i}(F_{j_1\cdots\hat{j_i}\cdots j_{k+1}}^{j_1\cdots j_{k+1}}+G_{j_1\cdots\hat{j_i}\cdots j_{k+1}}^{j_1\cdots j_{k+1}}) \ .
\end{equation}\par

A priori it cannot be guaranteed that the $g$'s belong to $C^\omega_\mathcal{F}(\CR^{2n})$, however, varying $j_1$ from $1$ to $n$, there is more than one decomposition like equations \ref{eqdeco1k+1} and \ref{eqdeco2k+1} for each combinations of vector fields. By the uniqueness of these decompositions (lemma \ref{Evalemma}) this yields $\alpha=\beta+\ud_\mathcal{F}\zeta$. There exists a correct number of functions to define $\beta$ and $\zeta$, the $g$'s and $F+G$'s of equations \ref{eqdeco1k+1} and \ref{eqdeco2k+1} (after applying uniqueness and identifying some of them, and using proposition \ref{welldefineform} and item 4 of lemma \ref{Evalemma} to guarantee that the forms are well defined).\par

The condition $\pounds_{X_i}(\beta)=0$ for all $i$ implies $\ud_\mathcal{F}\beta=0$, and there exists $\sigma\in\Omega_\mathcal{F}^{k-1}(\CR^{2n})$ such that $\ud_\mathcal{F}\sigma=\beta$ if and only if
\begin{equation}\label{betaeq1k}
\beta(X_{j_1},\dots,X_{j_k})=\somatorio{i=1}{k}(-1)^{i+1}X_{j_i}(\sigma(X_{j_1},\dots,\hat{X_{j_i}},\dots,X_{j_k})) \ .
\end{equation}\par

Applying lemma \ref{Evalemma},
\begin{equation}\label{sigmaeq1k}
\sigma(X_{j_1},\dots,\hat{X_{j_i}},\dots,X_{j_k})=s_{j_1\cdots\hat{j_i}\cdots j_k}^{j_1}+X_{j_1}(S_{j_1\cdots\hat{j_i}\cdots j_k}^{j_1}) \ ,
\end{equation}with $X_{j_1}(s_{j_1\cdots\hat{j_i}\cdots j_k}^{j_1})=0$. Now plugging equation \ref{sigmaeq1k} in \ref{betaeq1k} and  using the commutation of the vector fields ($[X_i,X_j]=0$) and because of  uniqueness (lemma \ref{Evalemma}), we obtain,
\begin{equation}
0=\beta(X_{j_1},\dots,X_{j_k})+\somatorio{i=2}{k}(-1)^{i}X_{j_i}(s_{j_1\cdots\hat{j_i}\cdots j_k}^{j_1})+X_{j_1}\left(\somatorio{i=1}{k}(-1)^{i}X_{j_i}(S_{j_1\cdots\hat{j_i}\cdots j_k}^{j_1})\right)
\end{equation}has solution if and only if,
\begin{equation}\label{betaeq2k}
0=\beta(X_{j_1},\dots,X_{j_k})+\somatorio{i=2}{k}(-1)^{i}X_{j_i}(s_{j_1\cdots\hat{j_i}\cdots j_k}^{j_1}) \ .
\end{equation}\par

Again, applying lemma \ref{Evalemma},
\begin{equation}\label{sigmaeq2k}
s_{j_1\cdots\hat{j_i}\cdots j_k}^{j_1}=s_{j_1\cdots\hat{j_i}\cdots j_k}^{j_1j_2}+X_{j_2}(S_{j_1\cdots\hat{j_i}\cdots j_k}^{j_1j_2}) \ ,
\end{equation}with $X_{j_2}(s_{j_1\cdots\hat{j_i}\cdots j_k}^{j_1j_2})=0$. Then, plugging equation \ref{sigmaeq2k} in \ref{betaeq2k}, using $[X_i,X_j]=0$ and invoking uniqueness (lemma \ref{Evalemma}),
\begin{equation}
0=\beta(X_{j_1},\dots,X_{j_k})+\somatorio{i=3}{k}(-1)^{i}X_{j_i}(s_{j_1\cdots\hat{j_i}\cdots j_k}^{j_1j_2})+X_{j_2}\left(\somatorio{i=2}{k}(-1)^{i}X_{j_i}(S_{j_1\cdots\hat{j_i}\cdots j_k}^{j_1j_2})\right)
\end{equation}has solution if and only if,
\begin{equation}
0=\beta(X_{j_1},\dots,X_{j_k})+\somatorio{i=3}{k}(-1)^{i}X_{j_i}(s_{j_1\cdots\hat{j_i}\cdots j_k}^{j_1j_2}) \ .
\end{equation}\par

Following the same procedure for all $X_{j_i}$, $i=3,\dots,k$, this yields $\beta(X_{j_1},\dots,X_{j_k})=0$.
\end{proof}

This determines all foliated cohomology groups in the analytic case,

\begin{coro}The foliated cohomology groups in the analytic case are determined for $k=1,\dots,n$ by,
\begin{eqnarray}
H_\mathcal{F}^k(\CR^{2n})&\cong&\displaystyle\bigoplus_{(j_1,\dots, j_k)}\{f_{j_1,\dots, j_k}\in C^\omega(\CR^{2n}) \nonumber \\
&;&  f_{j_1,\dots, j_k}(p)=f(h_1(p),\dots,h_n(p)) \text{and} \ f\big{|}_{\Sigma_{j_1}\cup\cdots\cup\Sigma_{j_k}}=0 \}
\end{eqnarray}where the right hand side has  $\left(\begin{array}{c}n \\ k\end{array}\right)$ summands.
\end{coro}
\begin{proof} Theorem \ref{cohomologycomputed} reads $H_\mathcal{F}^k(\CR^{2n})=W_\mathcal{F}^k(\CR^{2n})$, by definition any $\beta\in W_\mathcal{F}^k(\CR^{2n})$ can be given by $\left(\begin{array}{c}n \\ k\end{array}\right)$ functions vanishing at certain points (proposition \ref{welldefineform}) e.g.: $\beta(X_1,\dots,X_k)=f_{1\cdots k}\in C^\omega(\CR^{2n})$ and $f_{1\cdots k}=0$ at $\Sigma_1\cup\cdots\cup\Sigma_k$. The Lie derivative condition yields (item 5 of lemma \ref{Evalemma}) that each such function has a special dependence on its variables, e.g.: $f_{1\cdots k}(x_1,y_1,\dots,x_n,y_n)=f(h_1,\dots,h_n)$.
\end{proof}

The previous proofs work as well in the smooth category if all vector fields are of elliptic type. Thus for completely elliptic singularities we can compute all the cohomology groups obtaining the following:

\begin{teo}\label{ellipticcohomologycomputed}
{\normalfont{[\textbf{Elliptic case}]}} Consider $(\CR^{2n},\somatorio{i=1}{n}\ud x_i\wedge\ud y_i)$ with $h_1,\dots,h_n\in C^\infty(\CR^{2n})$ a Williamson basis. If all vector fields, $X_1,\dots,X_n$ are of elliptic type, for $k=1,\dots,n$:
\begin{equation}
\mathrm{ker}(\ud_\mathcal{F}:\Omega_\mathcal{F}^k(\CR^{2n})\to\Omega_\mathcal{F}^{k+1}(\CR^{2n}))=W_\mathcal{F}^k(\CR^{2n})\oplus\ud_\mathcal{F}(\Omega_\mathcal{F}^{k-1}(\CR^{2n})) \ ,
\end{equation}where $W_\mathcal{F}^k(\CR^{2n})$ is the set of $k$-forms $\beta\in\Omega_\mathcal{F}^k(\CR^{2n})$ such that $\pounds_{X_i}(\beta)=0$ for all $i$.

Thus, the foliated cohomology groups are given by:
\begin{eqnarray}
H_\mathcal{F}^k(\CR^{2n})&\cong&\displaystyle\bigoplus_{(j_1,\dots, j_k)}\{f_{j_1,\dots, j_k}\in C^\infty(\CR^{2n}) \nonumber \\
&;&  f_{j_1,\dots, j_k}(p)=f(h_1(p),\dots,h_n(p)) \text{and} \ f\big{|}_{\Sigma_{j_1}\cup\cdots\cup\Sigma_{j_k}}=0 \}
\end{eqnarray}where the right hand side has  $\left(\begin{array}{c}n \\ k\end{array}\right)$ summands.
\end{teo}



\end{document}